\documentclass[12pt,reqno]{amsart}
\usepackage{amssymb}
\usepackage{graphicx}
\usepackage{xcolor}
\usepackage{longtable}
\usepackage{float}

\usepackage[all]{xy}
\DeclareFontFamily{U}{mathb}{\hyphenchar\font45}
\DeclareFontShape{U}{mathb}{m}{n}{
      <5> <6> <7> <8> <9> <10> gen * mathb
      <10.95> mathb10 <12> <14.4> <17.28> <20.74> <24.88> mathb12
      }{}
\DeclareSymbolFont{mathb}{U}{mathb}{m}{n}
\DeclareMathSymbol{\righttoleftarrow}{3}{mathb}{"FD}

\theoremstyle{plain}
\newtheorem{prop}{Proposition}[section]
\newtheorem{theo}[prop]{Theorem}
\newtheorem{coro}[prop]{Corollary}
\newtheorem{lemm}[prop]{Lemma}
\theoremstyle{remark}

\theoremstyle{definition}

\newtheorem{rema}[prop]{Remark}

\newtheorem{exam}[prop]{Example}
\numberwithin{equation}{section}

\newcommand\cO{{\mathcal O}}

\newcommand\cT{{\mathcal T}}

\newcommand\bA{{\mathbb A}}
\newcommand\bG{{\mathbb G}}
\newcommand\bP{{\mathbb P}}
\newcommand\bQ{{\mathbb Q}}

\newcommand\bZ{{\mathbb Z}}
\newcommand\bR{{\mathbb R}}

\newcommand{\G}{{\mathbb G}}
\newcommand\rE{{\mathrm E}}
\newcommand\rH{{\mathrm H}}

\DeclareMathOperator{\image}{im}
\DeclareMathOperator{\coker}{coker}

\DeclareMathOperator{\Br}{Br}

\DeclareMathOperator{\Pic}{Pic}
\DeclareMathOperator{\Hom}{Hom}
\DeclareMathOperator{\End}{End}
\DeclareMathOperator{\Aut}{Aut}
\DeclareMathOperator{\Ext}{Ext}
\DeclareMathOperator{\Spec}{Spec}
\DeclareMathOperator{\Ker}{Ker}

\newcommand\lra{\longrightarrow}
\newcommand\GL{\mathsf{GL}}
\newcommand\PGL{\mathsf{PGL}}
\newcommand{\eqto}{\stackrel{\lower1.5pt\hbox{$\scriptstyle\sim\,$}}\to}

\begin{document}
\title[Equivariant unirationality]{Equivariant unirationality of toric varieties}

\author[A. Kresch]{Andrew Kresch}
\address{
  Institut f\"ur Mathematik,
  Universit\"at Z\"urich,
  Winterthurerstrasse 190,
  CH-8057 Z\"urich, Switzerland
}
\email{andrew.kresch@math.uzh.ch} 

\author[Yu. Tschinkel]{Yuri Tschinkel}
\address{
  Courant Institute,
  251 Mercer Street,
  New York, NY 10012, USA
}
\email{tschinkel@cims.nyu.edu}

\address{Simons Foundation\\
160 Fifth Avenue\\
New York, NY 10010\\
USA}

\date{June 6, 2025}

\begin{abstract}
We introduce a torsor-theoretic obstruction to equivariant unirationality and show that it is also sufficient for actions of finite groups on toric varieties arising from automorphisms of the torus. 
\end{abstract}

\maketitle

\section{Introduction}
\label{sec.intro}
Let $X$ be a smooth projective variety over an algebraically closed field $k$ of characteristic zero, equipped with a regular action of a finite group $G$. Such a variety if called $G$-unirational if there exists a dominant equivariant rational map to $X$ from the projective space associated with a linear representation of $G$.
This property should be viewed as the equivariant analog of unirationality of varieties over a not necessarily algebraically closed base field.
Both points of view have been pursued, starting with \cite{manin}. An explicit connection between equivariant geometry and geometry over nonclosed fields was established in \cite[Thm.\ 1.1]{DR}: A variety $X$ with regular $G$-action is 
$G$-unirational over $k$ if and only if for every field extension $K/k$ and every $G$-torsor $T$ over $K$, the $T$-twist of $X$ is unirational over $K$. However, this criterion is not easy to implement, in practice. Our focus is on computable obstructions to equivariant unirationality. 

The first and most obvious necessary condition for $G$-unirationality is 
\begin{itemize}
\item {\bf Condition (A)}:
for every abelian subgroup $H\subseteq G$ one has 
$$
X^H\neq \emptyset.
$$
\end{itemize}
Indeed, linear representations of abelian groups have fixed points, which forces fixed points on the image $X$.   
Note that Condition {\bf (A)} is a $G$-birational property, while the existence of $G$-fixed points is not, for nonabelian $G$. 

There are results showing that Condition {\bf (A)} is also sufficient, for regular generically free $G$-actions on  
del Pezzo surfaces of degree $\ge 3$ \cite[Thm.\ 1.4]{Duncan} and on
smooth quadric threefolds and intersections of two quadrics in $\bP^5$ \cite{CTZ-uni}.

Another necessary condition for $G$-unirationality is the existence of $G$-linearizations of line bundles with invariant class in $\Pic(X)$.
This cohomological obstruction, the Amitsur group, is recalled in Section \ref{sect:coho}, along with a higher analog whose vanishing is also necessary for $G$-unirationality.

In this note, we introduce a further obstruction to $G$-unirationality, 
inspired by the formalism of torsors over nonclosed fields \cite{CTSduke}, adapted to the equivariant setting in \cite{HT-torsor}.
The obstruction is formulated in terms of liftability of the $G$-action to a universal torsor and is equivalent to the nonvanishing of a certain explicit, computable cohomology class; see
\eqref{eqn.delIdzero} and Proposition~\ref{prop.delId}. We show that for toric varieties, with $G$-action  arising from automorphisms of the torus, 
the vanishing of this obstruction is also sufficient for $G$-unirationality (Theorem~\ref{thm:main}). 

After recalling background material in Section \ref{sect:coho} we discuss $G$-tori and their torsors in Section \ref{sec.Gtori} and
$G$-actions on toric varieties arising from 
equivariant compactifications of torsors under $G$-tori
in Section \ref{sect:toric}.
In Section \ref{sect:torsors}, we recall 
the universal torsor formalism, in the setting of $G$-equivariant geometry, and provide a recipe to compute the obstruction class to liftability of the $G$-action to a universal torsor.
The main theorem is stated and proved in Section \ref{sect:main}.
In Section \ref{scn.projectiveunirationality} we discuss the related condition of projective $G$-unirationality, introduced in \cite{tz}.
Several results make use of facts from homological algebra which, while standard, are included in an Appendix for the sake of completeness.

\medskip

\noindent
{\bf Acknowledgments:} 
The second author was partially supported by NSF grant 2301983.

\section{Equivariant geometry and cohomology}
\label{sect:coho}
Let $X$ be a smooth projective rational variety over $k$, with a regular action of a finite group $G$. Our convention is that $G$-actions on algebraic varieties are right actions.
Let $\Pic(X)$ be the Picard group of $X$, a $G$-module. Let $[X/G]$ be the quotient stack, and 
$$
\rH^1([X/G],\mathbb G_m) = \Pic(X,G)
$$
its Picard group, that can also be interpreted as the group of isomorphism classes of $G$-linearized line bundles on $X$. 

We recall the 
Leray spectral sequence for $G$-actions (see, e.g., \cite[\S 3]{KT-dp}):
\begin{align}
\begin{split}
\label{eqn:BrXG}
0&\to  \Hom(G,k^\times)\to \Pic(X,G)\to 
\Pic(X)^G \stackrel{\delta_2}{\lra}  \rH^2(G,k^\times)\\
&\qquad\quad \quad \to \Br([X/G])\to \rH^1(G,\Pic(X))\stackrel{\delta_3}{\lra}  \rH^3(G,k^\times).
\end{split}
\end{align} 
The sequence \eqref{eqn:BrXG} gives rise to the following invariants:
\begin{itemize}
\item The \emph{Amitsur group} $\mathrm{Am}(X,G)=\mathrm{Am}^2(X,G)$, defined to be the image of $\delta_2$ \cite[\S 6]{BC-finite};
\item The higher Amitsur group $\mathrm{Am}^3(X,G):=\mathrm{Im}(\delta_3)$ \cite[\S 3]{KT-dp}.  
\end{itemize}
The Amitsur group $\mathrm{Am}^2(X,G)$
is a stable $G$-birational invariant.
The same holds for $\rH^1(G,\Pic(X))$ as well as its
image
$\mathrm{Am}^3(X,G)$. 
By functoriality, the groups
$\mathrm{Am}^j(X,G)$, for $j=2$ and $3$, vanish
when $X$ has a fixed point.
From \cite{tz} we have:
\begin{prop}
\label{prop:obstr}
Let $Y\to X$ be a $G$-equivariant morphism of smooth projective varieties
with a regular $G$-action. Then 
$$
\mathrm{Am}^j(X,G)\subseteq \mathrm{Am}^j(Y,G), \quad j=2,3.
$$
\end{prop}

We recall basic conditions on the $G$-action.
These are formulated here, in a manner that allows $G$-actions with nontrivial generic stabilizer, though often our main interest is in generically free $G$-actions.
A $G$-action is
\begin{itemize}
\item[({\bf L})] {\em linearizable} if there exists a $G$-representation $V$ such that $X$ is equivariantly birational to $\bP(V^\vee)$,
\item[({\bf SL})] {\em stably linearizable} if $X\times \bP^n$ is linearizable, for some $n$, with trivial action on the second factor, 
\item[({\bf U})] {\em unirational} if there exists a representation $V$ of $G$ and a dominant equivariant rational map $\bP(V^\vee)\dashrightarrow X$. 
\end{itemize}

The following cohomological obstructions are applicable.
\begin{itemize}
\item If the $G$-action on $X$ is stably linearizable, then $\mathrm{Am}^j(X,H)=0$,
for $j=2$ and $3$ and  $\rH^1(H,\Pic(X))=0$, for all $H\subseteq G$.
\item If the action is unirational, then $\mathrm{Am}^j(X,H)=0$, $j=2$, $3$, for all $H\subseteq G$. 
\end{itemize}
We recall, as well, that Condition \textbf{(A)} (Section \ref{sec.intro}) is necessary for unirationality (hence also for stable linearizability).

\begin{exam}
\label{exam:p1}
Let $X=\bP^1$ with projective linear $G$-action $\lambda\colon G\to \PGL_2$ with image the Klein four-group $C_2^2$,
and let $H=\ker(\lambda)$.
The action is linearizable if and only if $[G,H]$ is strictly contained in $[G,G]$; otherwise it has nontrivial Amitsur group and thus is not unirational.
(Use an extension to $H/[G,H]$ of the nontrivial character of $[G,G]/[G,H]\cong C_2$ to linearize $\lambda$.)
\end{exam}

\begin{rema} 
\label{rema:possible}
An example of a del Pezzo surface of degree $3$ with nontrivial $\mathrm{Am}^3(X,G)$ is given in \cite[\S 5.3]{KT-dp}; 
this example fails Condition {\bf (A)}.
A complete analysis of this invariant for del Pezzo surfaces can be found in \cite{tz}; there exist del Pezzo surfaces of degree $2$ satisfying Condition {\bf (A)}, with nontrivial $\mathrm{Am}^3(X,G)$ for $G=\mathfrak{Q}_8$, the quaternion group.  
\end{rema}

\begin{rema}
The groups $\mathrm{Am}^j(X,G)$, for $j=2$ and $3$, are effectively computable, starting from the knowledge of a $G$-invariant collection of divisors, spanning $\Pic(X)$, and their relations.
The computation is described in \cite{KT-Brauer} and is also recalled in Section \ref{sect:toric}. 
\end{rema}

\begin{rema}
\label{rem.uniratPic}
Unirational actions with nontrivial $\rH^1(G,\Pic(X))$ exist, for instance, among del Pezzo surfaces, by the main results of \cite{Duncan} and \cite{BogPro}, or toric varieties, for which existence of a $G$-fixed point implies $G$-unirationality, by \cite[\S 3.6, \S 4.2]{HT-torsor}. 
\end{rema}

\section{$G$-tori and cohomology}
\label{sec.Gtori}
We continue to work over an algebraically closed field $k$ of characteristic zero. 

Let $S\cong \bG_m^n$ be an algebraic torus,
with character lattice $L=\mathfrak X^*(S)$.
Let $G$ be a finite group.
A $G$-torus is a torus $S$ with $G$-action by automorphisms that fix $1\in S$ (but we do not assume that the action is generically free). 
Associated with this action is an
induced left action
\[ G \to \GL(L). \]
Conversely, such a homomorphism allows us to recover a structure of a $G$-torus, via $S\cong L^\vee\otimes \G_m$.

To a $G$-torus $S$ there is an associated
sheaf of abelian groups on the
\'etale site of the Deligne-Mumford stack $BG$,
which we also denote by $S$.
We have
\[ \rH^i(BG,S)=\rH^i(G,S(k)) \]
(consequence of the \v Cech spectral sequence for $\Spec(k)\to BG$).

Let $X$ be a smooth projective variety with a regular action of $G$; then $S$ will also denote the corresponding sheaf of abelian groups on the \'etale site of the Deligne-Mumford stack $[X/G]$.
We introduce the notation
\[ \rH^i_G(X,S):=\rH^i([X/G],S). \]
For $i=1$, this has an interpretation as the group of isomorphism classes of $G$-equivariant $S$-torsors on $X$.
However, there is no direct description in terms of group cohomology.
Rather, the Leray spectral sequence
\[ \rE_2^{pq}=\rH^p(G,\rH^q(X,S))\Rightarrow \rH^{p+q}_G(X,S) \]
involves a mixture of group cohomology and classical sheaf cohomology, a phenomenon already observed in a slightly different context in \cite[\S 4]{DM}.

We have
$\rH^q(X,S)=L^\vee\otimes\rH^q(X,\G_m)$.
The low-degree terms form the exact sequence
\begin{align}
\begin{split}
\label{eqn.LPic}
0&\to \rH^1(G,L^\vee\otimes k^\times)\to
\rH^1_G(X,S) \\
&\quad\quad\quad\to (L^\vee\otimes\Pic(X))^G\stackrel{\partial}\to \rH^2(G,L^\vee\otimes k^\times)\to \rH^2_G(X,S),
\end{split}
\end{align}
see also \cite[\S 3.4]{HT-torsor}. 

\begin{rema}
\label{rem.fixedpoint}
If $X$ has a fixed point for the $G$-action, then by basic functoriality the rightmost map in \eqref{eqn.LPic} is injective, thus $\partial$ is trivial.
\end{rema}

\section{Toric varieties}
\label{sect:toric}
Let $X$ be a smooth projective equivariant compactification of an algebraic torus $T\simeq \bG_m^d$ with a regular action of a finite group $G$.
Let $M:=\mathfrak X^*(T)$.
We have a basic exact sequence of torsion-free $G$-modules
\begin{equation}
\label{eqn.MPPic}
0\to M\to P \to \Pic(X)\to 0,  
\end{equation}
where $P$ is the permutation module of toric divisors. We also have
\begin{equation}
\label{eqn.kkTM}
1\to k^\times \to k[T]^\times \to M\to 0. 
\end{equation}
Combining these sequences we obtain the diagrams
\[
\xymatrix@C=12pt@R=15pt{
&&
\rH^1(G, k[T]^\times)\ar[d] & \\
P^G \ar[r] &
\Pic(X)^G \ar[r]\ar[dr]_{\delta_2} &
\rH^1(G,M) \ar[r] \ar[d] &
0 \\
&&\rH^2(G,k^\times)&
}
\]
and
\[
\xymatrix@C=12pt@R=15pt{
&&
\rH^2(G, k[T]^\times)\ar[d] & \\
0 \ar[r] &
\rH^1(G,\Pic(X)) \ar[r]\ar[dr]_{\delta_3} &
\rH^2(G,M) \ar[r] \ar[d] &
\rH^2(G,P) \\
&&\rH^3(G,k^\times)&
}
\]
with exact rows and columns.
The triangle in each diagram commutes, by Lemma \ref{lem.twoterm}.
The second diagram already appears in \cite[\S 4]{KT-Brauer}.

The case where $T$ has a $G$-fixed point, without loss of generality $1\in T$, is that of a
\emph{toric action}, where $X$ is an equivariant compactification of a $G$-torus;
these are classified in small dimensions \cite{kunyavskii}, \cite{hoshiyamasaki}, and have also been studied from the viewpoint of equivariant birational geometry \cite{KT-toric}.
In this case, sequence \eqref{eqn.kkTM} admits an equivariant splitting, and $\delta_2$ and $\delta_3$ vanish.
By way of contrast, here we allow the more general actions considered in \cite{HT-torsor}, where $X$ is an equivariant compactification of a torsor under a $G$-torus.

\begin{exam}
\label{exa.U11withsigns}
We verify the nontriviality of $\mathrm{Am}^3(X,G)$ for $X$ an equivariant compactification of $T=\G_m^3$, with action of $G:=\mathfrak{Q}_8$, where respective generators $x$, $y\in G$ act by
\[ (\alpha,\beta,\gamma)\mapsto \bigl(-\frac{1}{\alpha\beta\gamma},-\gamma,\beta\bigr),\quad
\bigl(\alpha,\beta,\gamma)\mapsto (\gamma,-\frac{1}{\alpha\beta\gamma},-\alpha\bigr). \]
We take $X$ to be the smooth projective toric variety, for a fan in
$N_{\bR}\simeq \bR^3$, $N:=\mathfrak{X}_*(T)$ with ray generators
\begin{gather*}
e_1, e_2, e_3, -e_1, -e_2, -e_3, \\
e_2-e_1, e_1-e_2, e_3-e_1, e_1-e_3, e_3-e_2, e_2-e_3, \\
e_1+e_2-e_3, e_1+e_3-e_2, e_2+e_3-e_1, \\
-e_1-e_2+e_3, -e_1-e_3+e_2, -e_2-e_3+e_1.
\end{gather*}
Just the ray generators on the first two lines yield a fan with $14$ three-dimensional cones, $8$ simplicial and $6$ with four generators;
cf.\ \cite[\S 2, Case P]{kunyavskii} (where $e_i$ is called $g_i$).
The rays from the last two lines subdivide the
cones with four generators.
Without the minus signs in the formula for the action, we would have the toric action given in \cite{kunyavskii} as $U_1$ (with respect to coordinates $\alpha'=\alpha\beta$, $\beta'=\alpha\beta\gamma$, $\gamma'=\alpha$).
But the signs play no role for the
action on $M$, on the permutation module of toric divisors, or on $\Pic(X)\cong \bZ^{15}$, thus from \cite[\S 3]{kunyavskii} we obtain
\[ \rH^1(G,\Pic(X))=\rH^1(G/\langle x^2\rangle, \Pic(X))\cong \bZ/2\bZ. \]
The abelian subgroups of $G$ are cyclic, so the action satisfies Condition \textbf{(A)}, as does indeed any action of $G$ on a rational variety.

For computations in group cohomology we use the periodic resolution given in \cite[\S XII.7]{cartaneilenberg},
which yields $\rH^j(G,R)$ as the $j$th cohomology of
\[ R\stackrel{\begin{pmatrix}\scriptstyle 1-x\\ \scriptstyle 1-y\end{pmatrix}}{\relbar\joinrel\longrightarrow}
R^2\stackrel{\setlength\arraycolsep{2pt}\begin{pmatrix}\scriptstyle 1+x&\scriptstyle -1-y\\ \scriptstyle 1+xy&\scriptstyle -1+x\end{pmatrix}}{\relbar\joinrel\relbar\joinrel\relbar\joinrel\relbar\joinrel\longrightarrow}
R^2\stackrel{\setlength\arraycolsep{2pt}\begin{pmatrix}\scriptstyle 1-x&\scriptstyle -1+xy \end{pmatrix}}{\relbar\joinrel\relbar\joinrel\relbar\joinrel\relbar\joinrel\longrightarrow}
R\stackrel{\sum_{g\in G}g}{\relbar\joinrel\longrightarrow}
R \stackrel{\begin{pmatrix}\scriptstyle 1-x\\ \scriptstyle 1-y\end{pmatrix}}{\relbar\joinrel\longrightarrow} \dots. \]
In particular, $\rH^2(G,k^\times)=0$ and
$\rH^3(G,k^\times)\cong \bZ/8\bZ$, so
$\delta_2$ is trivial and for $\delta_3$ there is just one nontrivial possibility.

We denote the toric divisors by $D_1$, $D_2$, $D_3$ (corresponding to the standard basis elements of $N$), $\widehat{D}_1$, $\widehat{D}_2$, $\widehat{D}_3$ (corresponding to their negatives), $D_{21}$, $D_{12}$, $D_{31}$, $D_{13}$, $D_{32}$, $D_{23}$ (corresponding to $e_2-e_1$, etc.),
$D_{123}$, $D_{132}$, $D_{231}$ (corresponding to $e_1+e_2-e_3$, etc.),
$\widehat{D}_{123}$, $\widehat{D}_{132}$, $\widehat{D}_{231}$ (corresponding to their negatives).
We have $\Pic(X)$ as quotient of the corresponding permutation module $P=\bZ^{18}$ by the relations in Table \ref{relPicX}.
A cocycle representative for the nonzero element of $\rH^1(G,\Pic(X))$ is
\[ (0,-[D_1]+[D_{31}]-[D_{13}]+[D_{32}]-[D_{123}]+[\widehat{D}_{123}]). \]
To the evident divisor lift we apply the differential of the complex that computes $\rH^*(G,P)$ and obtain the pair with first component $0$ and second component the sum of the left-hand sides of the first two entries in Table \ref{relPicX}.
Thus, under the connecting homomorphism to $\rH^2(G,M)$ the nonzero element of $\rH^1(G,\Pic(X))$ maps to the class represented by
\[ (0,e_1^\vee+e_2^\vee). \]
We lift to $(1,\alpha\beta)\in (k[T]^\times)^2$ and apply
$\setlength\arraycolsep{3pt}\begin{pmatrix}1-x&-1+xy\end{pmatrix}$ to get
$-1\in k^\times$.
So $\delta_3$ is nontrivial, and $\mathrm{Am}^3(X,G)\cong \bZ/2\bZ$.
\end{exam}

\begin{table}[H]
\small
\begin{align*}
&D_1-\widehat{D}_1-D_{21}+D_{12}-D_{31}+D_{13}
+D_{123}+D_{132}-D_{231}\\
&\qquad\qquad\qquad\qquad\qquad\qquad\qquad\qquad-\widehat{D}_{123}-\widehat{D}_{132}+\widehat{D}_{231}=0\\
&D_2-\widehat{D}_2+D_{21}-D_{12}-D_{32}+D_{23}
+D_{123}-D_{132}+D_{231}\\
&\qquad\qquad\qquad\qquad\qquad\qquad\qquad\qquad-\widehat{D}_{123}+\widehat{D}_{132}-\widehat{D}_{231}=0\\
&D_3-\widehat{D}_3+D_{31}-D_{13}+D_{32}-D_{23}
-D_{123}+D_{132}+D_{231}\\
&\qquad\qquad\qquad\qquad\qquad\qquad\qquad\qquad+\widehat{D}_{123}-\widehat{D}_{132}-\widehat{D}_{231}=0
\end{align*}
\caption{Relations in $\Pic(X)$}
\label{relPicX}
\end{table}

\section{Equivariant universal torsors}
\label{sect:torsors}
Let $X$ be a smooth projective rational variety
with action of $G$ and
$T_{\mathrm{NS}}$ the N\'eron-Severi torus, i.e., the $G$-torus with character lattice $\Pic(X)$.
We have the spectral sequence from Section \ref{sec.Gtori} and the exact sequence \eqref{eqn.LPic}, which includes the terms
\begin{equation}
\label{eqn.Thatdel}
\rH^1_G(X,T_{\mathrm{NS}})\to
\End_G(\Pic(X))\stackrel{\partial}\to
\rH^2(G,\Pic(X)^\vee\otimes k^\times).
\end{equation}
A \emph{$G$-equivariant universal torsor} is an equivariant $T_{\mathrm{NS}}$-torsor on $X$ whose class in $\rH^1_G(X,T_{\mathrm{NS}})$ maps to
$1_{\Pic(X)}\in \End_G(\Pic(X))$.
By \eqref{eqn.Thatdel}, 
there exists a $G$-equivariant universal torsor on $X$ if and only if 
\begin{equation}
\label{eqn.delIdzero}
\partial(1_{\Pic(X)}) = 0\in \rH^2(G,\Pic(X)^\vee\otimes k^\times). 
\end{equation}

For instance, by Remark \ref{rem.fixedpoint}, if $X$ has a fixed point then Condition \eqref{eqn.delIdzero} is satisfied.

\begin{prop}
\label{prop.delId}
Condition \eqref{eqn.delIdzero} is
a stable $G$-birational invariant.
Furthermore, if $X$ is $G$-unirational, then Condition \eqref{eqn.delIdzero} is satisfied.
\end{prop}

\begin{proof}
The proof uses the functoriality of the exact sequence \eqref{eqn.LPic}, in the character lattice as well as in $X$; cf.\ \cite[Prop.\ 1.5.2 (i)]{CTSduke}.

For $G$-birational invariance it suffices, by equivariant weak factorization (cf.\ \cite{BnG}), to establish invariance under blow-up
$\pi\colon \widetilde{X}\to X$ along a smooth $G$-invariant center.
Let us denote by $E_1$, $\dots$, $E_r$ the components of the exceptional divisor.
We have $\Pic(\widetilde{X})=\Pic(X)\oplus W$,
with $W=\bZ\cdot [E_1]\oplus \dots\oplus \bZ\cdot [E_r]$ and permutation action on $W$.
By functoriality we have a morphism from the exact sequence \eqref{eqn.Thatdel} to the sequence \eqref{eqn.LPic} for $\Pic(X)$ on $\widetilde{X}$;
denoting by $\partial'$ the $E_2$-differential appearing in the latter, we obtain
\[ \partial'(\pi^*)=\partial(1_{\Pic(X)}). \]
In particular, under the decomposition
\[ \rH^2(G,\Pic(\widetilde{X})^\vee\otimes k^\times)=\rH^2(G,\Pic(X)^\vee\otimes k^\times)\oplus \rH^2(G,W^\vee\otimes k^\times), \]
the class $\partial(1_{\Pic(\widetilde{X})})$ has first component $\partial(1_{\Pic(X)})$.
A torsor, consisting of
nonvanishing sections
of the line bundles $\cO_{\widetilde{X}}(E_i)$ for $i=1$, $\dots$, $r$, exhibits the vanishing of the second component of $\partial(1_{\Pic(\widetilde{X})})$ and thus the $G$-birational invariance of
Condition \eqref{eqn.delIdzero}.
To strengthen this to stable birational invariance we use
a similar combination of functoriality and the evident $\bG_m$-torsor on projective space.

For the second assertion, by $G$-birational invariance we are reduced to showing, for an equivariant morphism $\pi\colon Y\to X$, that Condition \eqref{eqn.delIdzero} for $Y$ implies the same for $X$.
We have $\pi^*\colon \Pic(X)\to \Pic(Y)$, inducing the downward maps in the commutative diagram
\[
\xymatrix{
\End_G(\Pic(Y)) \ar[r] \ar[d] &
\rH^2(G,\Pic(Y)^\vee\otimes k^\times) \ar[d] \\
\Hom_G(\Pic(X),\Pic(Y))\ar[r] &
\rH^2(G,\Pic(X)^\vee\otimes k^\times) \\
\End_G(\Pic(X))\ar[u]\ar[r] & \rH^2(G,\Pic(X)^\vee\otimes k^\times)\ar@{=}[u]
}
\]
A straightforward diagram chase gives the implication.
\end{proof}

\begin{rema}
\label{rem.product}
Let $X_1$ and $X_2$ be smooth projective rational $G$-varieties.
Then Condition \eqref{eqn.delIdzero} holds for $X_1\times X_2$, if and only if it holds for $X_1$ and $X_2$ individually.
By functoriality, as in the proof of Proposition \ref{prop.delId}, under the decomposition
$$
\rH^2(G,\Pic(X_1\times X_2)^\vee\otimes k^\times)= \rH^2(G,\Pic(X_1)^\vee\otimes k^\times)\oplus  \rH^2(G,\Pic(X_2)^\vee\otimes k^\times)
$$
we have 
$$
\partial(1_{\Pic(X_1\times X_2)})=
(\partial(1_{\Pic(X_1)}),\partial(1_{\Pic(X_2)})).
$$
\end{rema}

\begin{rema}
\label{rem.sequencesplits}
Proposition 4 of \cite{HT-torsor}, used there only in the proof of $G$-birational invariance (Proposition \ref{prop.delId}, here), is incorrect as stated.
We consider $X=\bP^1$, with $U=\bA^1\setminus \{0\}$ and $G$ of order $2$, acting by $x\mapsto -x$. 
Now $T_{\mathrm{NS}}=\bG_m$, and
$$
\rH^2(G,\bG_m)=0.
$$
There exists a $G$-equivariant universal torsor over $X$, by \eqref{eqn.Thatdel} or equally well by the fact that $X$ has fixed points.
However, the sequence 
$$
1\to k^\times \to k[U]^\times \to k[U]^\times/k^\times \to 1
$$
does \emph{not} admit a $G$-equivariant splitting.
Indeed, writing $\alpha x^j$ as $(\alpha,j)$ to get $k[U]^\times\cong k^\times\times\bZ$,
the action becomes
$$
(\alpha,j)\mapsto ((-1)^j\alpha, j),
$$
with map to $k[U]^\times/k^\times\simeq \bZ$ given by projection to the second factor, and
there is no invariant lift of $1\in \bZ$.
Similarly, we may see that the sequence
\[ 1\to k^\times\to k(X)^\times\to k(X)^\times/k^\times\to 1, \]
analogous to the one used in \cite[\S 2.2]{CTSduke} for a Galois action to define the elementary obstruction, does not admit a $G$-equivariant splitting.
%
\end{rema}

The class
\begin{equation}
\label{eqn.wanttocompute}
\partial(1_{\Pic(X)})\in \rH^2(G,\Pic(X)^\vee\otimes k^\times)
\end{equation}
is the $G$-equivariant analog of the universal obstruction class $\partial(\lambda_0)$, considered in the context of geometry over nonclosed fields in \cite[Prop.\ 2.2.8 (i)--(iii)]{CTSduke}.
However, in the equivariant context, it is {\em not} equivalent to the elementary obstruction, as shown in Remark \ref{rem.sequencesplits}.

Now we explain how to compute the class \eqref{eqn.wanttocompute}
in concrete geometric situations. Let $U\subset X$ be a Zariski open subset, whose boundary is a $G$-stable union of irreducible divisors $D=\bigcup_{i\in I} D_i$, generating $\Pic(X)$.  
We let $P=\bigoplus_i \bZ\cdot D_i$ denote the corresponding permutation module and $M$ the module of relations, so that we have an exact sequence \eqref{eqn.MPPic}.
Dualizing and tensoring with $k^\times$ we obtain
\begin{equation}
\label{eqn.MPPicdual}
0\to \Pic(X)^\vee\otimes k^\times\to P^\vee\otimes k^\times\to M^\vee\otimes k^\times\to 0,
\end{equation}
which is exact since $\Pic(X)$ is torsion-free.

We also have the exact sequence
\begin{equation}
\label{eqn.kkUM}
1\to k^\times \to k[U]^\times \to M\to 0.
\end{equation}
The extension class
\[
\rho\in \rH^1(G,M^\vee\otimes k^\times)
\]
of \eqref{eqn.kkUM}
may, following Lemma \ref{lem.ABCidentity}, be computed by applying the differential of a complex that computes group cohomology to $M^\vee\otimes k[U]^\times$ to get a $1$-cocycle with values in $M^\vee\otimes k^\times$,
that represents $-\rho$;
for computation with explicit cocycles one should be aware of the signs from \S \ref{ss.groupcohomology}.
By an analogous computation, from a $1$-cocycle representative of $\rho$ we obtain a $2$-cocycle representative of the image
\[ \sigma\in \rH^2(G,\Pic(X)^\vee\otimes k^\times). \]
under the connecting homomorphism of group cohomology of \eqref{eqn.MPPicdual}.

\begin{prop}
\label{prop.delidentity}
With the above notation we have
\[ \partial(1_{\mathrm{\Pic(X)}})=-\sigma. \]
\end{prop}

\begin{proof}
Tensoring the previous exact sequences by $\Pic(X)^\vee$, we have
\begin{gather*}
0\to \Pic(X)^\vee\otimes  M\to \Pic(X)^\vee\otimes P \to \End(\Pic(X))\to 0, \\
0\to \Pic(X)^\vee\otimes k^\times \to \Pic(X)^\vee\otimes k[U]^\times \to \Pic(X)^\vee\otimes M\to 0.
\end{gather*}
The exact sequences combine to
\[
\xymatrix@C=12pt@R=15pt{
&&
\rH^1(G, \Pic(X)^\vee\otimes k[U]^\times)\ar[d] & \\
(\Pic(X)^\vee\otimes P)^G \ar[r] &
\End_G(\Pic(X)) \ar[r]\ar[dr]_{\partial} &
\rH^1(G,\Pic(X)^\vee\otimes M) \ar[d] \\
&&\rH^2(G,\Pic(X)^\vee\otimes k^\times)&
}
\]
where the triangle commutes, by Lemma \ref{lem.twoterm}.
By Lemma \ref{lem.ABC2}, we have the anticommuting square
\[
\xymatrix@C=18pt@R=15pt{
\End_G(\Pic(X))\ar[dr]&\End_G(M) \ar[r] \ar[d] & \rH^1(G,M^\vee\otimes k^\times) \ar[d] \\
& \rH^1(G,\Pic(X)^\vee\otimes M) \ar[r] & \rH^2(G,\Pic(X)^\vee\otimes k^\times)
}
\]
in the diagram of connecting homomorphisms of group cohomology, and by Lemma \ref{lem.ABCidentity}, the images of $1_{\Pic(X)}$ and $1_M$ are inverse to each other.
A diagram chase, with attention to signs, yields the result.
\end{proof}

\begin{rema} 
If the $G$-action on $\Pic(X)$ is trivial, then $\partial(1_{\Pic(X)})=0$ if and only if $\mathrm{Am}^2(X,G)=0$.
We choose a basis $\Pic(X)\cong \bZ^r$ and compare
\eqref{eqn.Thatdel}
with \eqref{eqn:BrXG} via
the $r$ projections
$\G_m^r\to \G_m$.
Then $\delta_2$ of the basis gives
$\partial(1_{\Pic(X)})\in \rH^2(G,\Pic(X)^\vee\otimes k^\times)\cong \rH^2(G,k^\times)^r$.
\end{rema}

\section{Unirationality of actions on toric varieties}
\label{sect:main}
We apply the formalism of Section~\ref{sect:torsors} to establish a criterion for $G$-unirationality of actions on toric varieties. 

We return to the setup of Section~\ref{sect:toric}: let  $T\simeq \bG_m^d$ be an algebraic torus, with character lattice $M:=\mathfrak{X}^*(T)$. From Section~\ref{sec.Gtori}, automorphisms of $T$ as an algebraic torus (i.e., fixing $1\in T$) are identified with $\GL(M)$.
Automorphisms of $T$ as an algebraic variety admit a description as semidirect product:
\begin{equation}
\label{eqn.semidirect}
1\to T(k)\to \Aut(T)\to \GL(M)\to 1. 
\end{equation}

Let $X$ be a smooth projective toric variety, which is an equivariant compactification of $T$.
Then $X$ arises from a smooth projective $G$-invariant fan in 
$N_{\mathbb{R}}$, where $N=\Hom(M,\bZ)$ is
the cocharacter lattice; see \cite{CTHS} for the existence of such a fan.
The $1$-dimensional cones of the fan give a permutation module $P$; they correspond to the toric divisors. 

We recall the map
\[
\partial\colon \End_G(\Pic(X))\to
\rH^2(G,\Pic(X)^\vee\otimes k^\times)
\]
from \eqref{eqn.Thatdel}
and Condition \eqref{eqn.delIdzero}, vanishing of $\partial(1_{\Pic(X)})$, which generally for an action of a finite group on a smooth projective rational variety is necessary for $G$-unirationality (Proposition \ref{prop.delId}).

Our main theorem establishes the sufficiency, for $X$ as considered here.

\begin{theo}
\label{thm:main}
Let $X$ be a smooth projective toric variety with a regular action by a finite group $G$ such that the torus $T\subset X$ is $G$-stable.
Then $X$ is $G$-unirational if and only if
$\partial(1_{\Pic(X)})=0$.
\end{theo}

Before we embark on the proof we recall a cohomological characterization of an action as in Theorem \ref{thm:main}.

Let a homomorphism $G\to \mathsf{GL}(M)$ be given,
determining a toric action of $G$ on $X$:
\[ (a_g)_{g\in G},\qquad a_g\in \Aut(X),\qquad a_{g'}\circ a_g=a_{gg'},\,\,\forall\,g,g'\in G. \]
In the setting of Theorem \ref{thm:main} this arises from the given action by application of the right-hand map in \eqref{eqn.semidirect}.
We characterize the actions of $G$ on $X$, leaving $T$ stable, that lift the toric action $(a_g)_g$.
Such an action is determined by
$(\lambda_g)_g$ with $\lambda_g\in T(k)$ for $g\in G$,
where we let 
\[ \tau_g\in \Aut(X) \]
denote translation by $\lambda_g$
and associate to $(a_g)_g$ and $(\lambda_g)_g$ the collection of automorphisms
$(a_g\circ \tau_g)_g$.
The condition for $(a_g\circ \tau_g)_g$ to be an action is
\begin{equation}
\label{eqn.withtranslation}
a_{g'}\circ \tau_{g'}\circ a_g\circ \tau_g=a_{gg'}\circ \tau_{gg'}, \quad \,\,\forall\,g,g'\in G.
\end{equation}
As may be checked directly, \eqref{eqn.withtranslation}
is equivalent to
$\lambda_g a_g^{-1}(\lambda_{g'})=\lambda_{gg'}$, which is the cocycle condition for $(\lambda_g)_g$, where $T(k)$ is endowed with the $G$-module structure $g\cdot \zeta=a_g^{-1}(\zeta)$.
We regard two such modified actions as equivalent, if one transforms to the other under the identification of $X$ with itself by translation by some $\zeta\in T(k)$.
The effect on $a_g\circ \tau_g$ is to compose with translation by $a_g^{-1}(\zeta)$, i.e., change $(\lambda_g)_g$ by a coboundary.
So the equivalence classes of actions, lifting the toric action given by $(a_g)_g$, are classified by
\[ \rH^1(G,M^\vee\otimes k^\times). \]

In the setting of Theorem \ref{thm:main}, we let the given action correspond to $(a_g)_g$ and $(\lambda_g)_g$ as above, with
$[(\lambda_g)_g]\in \rH^1(G,M^\vee\otimes k^\times)$.
We look at the two basic exact sequences from Section \ref{sect:toric}.
The sequence \eqref{eqn.MPPic} is sensitive only to the toric action and does not see the effect of the translations.
The other sequence \eqref{eqn.kkTM}, we claim, has the extension class
\begin{equation}
\label{eqn.explicitrho}
\rho=[(\lambda_g)_g]\in \rH^1(G,M^\vee\otimes k^\times).
\end{equation}
To see this, we apply the recipe from Section \ref{sect:torsors}.
The action of $g\in G$ sends $(\chi\mapsto \chi)\in M^\vee\otimes k[T]^\times$ to
$(\chi\mapsto \lambda_g(\chi)\chi)$, so we get the cocycle representative $(\lambda_g^{-1})_g$ of $-\rho$.

\begin{proof}[Proof of Theorem \ref{thm:main}]
Proposition \ref{prop.delId} supplies the forwards implication.
For the reverse implication we let $\sigma\in \rH^2(G,\Pic(X)^\vee\otimes k^\times)$ be obtained from $\rho$, above, by the connecting homomorphism of the exact sequence \eqref{eqn.MPPicdual}.
By Proposition \ref{prop.delidentity},
we have $\partial(1_{\Pic(X)})=-\sigma$.
If this vanishes, then $\rho$ admits a lift to a class in $\rH^1(G,P^\vee\otimes k^\times)$.
By an argument similar to the one above, a cocycle representative determines a modification by translations of the right permutation action on the torus $\Spec(k[P])$.
This is a linear action, so $X$ is $G$-unirational.
\end{proof}

\begin{rema}
\label{rema.aboutU11withsigns}
Example~\ref{exa.U11withsigns} is the unique (up to equivalence) nontrivial modification by translations of the stated toric action of $G=\mathfrak{Q}_8$.
Indeed, $\rH^1(G,M^\vee\otimes k^\times)\cong\bZ/2\bZ$:
the complex for group cohomology of $G$ yields
the quotient by
all $(\alpha^2\beta\gamma,\beta\gamma^{-1},\beta^{-1}\gamma,\alpha\gamma^{-1},\alpha\beta^2\gamma,\alpha^{-1}\gamma)$
of the
$(\alpha_1,\beta_1,\gamma_1,\alpha_2,\beta_2,\gamma_2)\in T(k)^2$ with
$\beta_1\gamma_1\alpha_2\gamma_2=1$,
$\beta_1\gamma_1\alpha_2^{-1}\gamma_2^{-1}=1$,
$\alpha_1\beta_1\beta_2^{-1}\gamma_2=1$.
The first two conditions lead to
$\beta_1\gamma_1=\alpha_2\gamma_2\in \mu_2$,
and the remaining one gives
$\beta_2=\alpha_1\beta_1\gamma_2$.
Given $\alpha_1$, $\beta_1$, $\gamma_2\in k^\times$, we take $\beta$ to be a $4$th root of $\alpha_1\beta_1^3\gamma_2^2$,
and with $\alpha=\beta_1^{-1}\gamma_2^{-1}\beta$, $\gamma=\beta_1^{-1}\beta$, we have
\[
(\alpha^2\beta\gamma,\beta\gamma^{-1},\beta^{-1}\gamma,\alpha\gamma^{-1},\alpha\beta^2\gamma,\alpha^{-1}\gamma)=(\alpha_1,\beta_1,\beta_1^{-1},\gamma_2^{-1},\alpha_1\beta_1\gamma_2,\gamma_2).
\]
\end{rema}

\begin{exam}
\label{exa.noAm3}
Practical implemention of the criterion from Theorem \ref{thm:main} quickly leads to large computations.
To give the reader a flavor, we sketch a $5$-dimensional example for the dihedral group $G:=\mathfrak{D}_4$.
For this group Condition \textbf{(A)} is not automatic
($G$ has non-cyclic abelian subgroups),
but does imply $\mathrm{Am}^j(X,G)=0$ for $j=2$, $3$, since
restriction to the three subgroups of index $2$ induces injective
\[ \rH^j(G,k^\times)\to \rH^j(C_4,k^\times)\oplus \rH^j(C_2^2,k^\times)\oplus \rH^j(C_2^2,k^\times). \]
Let $X$ be an equivariant compactification of
$T=\G_m^5$ with generators $x$ of order $4$ and $y$ of order $2$ acting by
\[ (\alpha,\beta,\gamma,\delta,\varepsilon)\mapsto \bigl(\alpha\beta,\frac{1}{\beta},\frac{\delta}{\gamma},-\delta,i\varepsilon\bigr),\quad
(\alpha,\beta,\gamma,\delta,\varepsilon)\mapsto \bigl(\alpha,\beta,\frac{\beta}{\gamma},-\frac{1}{\delta},-i\delta\varepsilon\bigr),
\]
where the underlying toric action has been taken from \cite[Lemma 6.2]{hellerreinerI}.
For group cohomology $\rH^j(G,R)$ we use the complex ($N_x:=1+x+x^2+x^3$)
\[ R\stackrel{\begin{pmatrix}\scriptstyle 1-x\\ \scriptstyle 1-y\end{pmatrix}}\longrightarrow
R^2\stackrel{\setlength\arraycolsep{2pt}\begin{pmatrix}\scriptstyle N_x&\scriptstyle 0\\ \scriptstyle 1+xy&\scriptstyle -1+x\\ \scriptstyle 0&\scriptstyle 1+y\end{pmatrix}}{\relbar\joinrel\relbar\joinrel\relbar\joinrel\longrightarrow}
R^3\stackrel{\setlength\arraycolsep{2pt}\begin{pmatrix}\scriptstyle 1-x&\scriptstyle 0&\scriptstyle0\\ \scriptstyle 1+y&\scriptstyle -N_x&\scriptstyle0\\ \scriptstyle0&\scriptstyle 1-xy&\scriptstyle 1-x\\ \scriptstyle0&\scriptstyle0&\scriptstyle 1-y\end{pmatrix}}{\relbar\joinrel\relbar\joinrel\relbar\joinrel\relbar\joinrel\relbar\joinrel\longrightarrow}
R^4\dots,
\]
cf.\ \cite[\S IV.2]{ademmilgram}.
A compactification, equivariant and with evident fixed points for the subgroup $\langle x^2,y\rangle$, is the copy of $(\bP^1)^5$ with coordinate charts $k[\alpha^{\pm 1},(\beta\gamma^{-1})^{\pm 1},\gamma^{\pm 1},(\delta\varepsilon)^{\pm 1},\varepsilon^{\pm 1}]$,
but to be $G$-invariant the set of rays needs to be augmented as in the top $5$ lines of Table \ref{rays}.
We add more rays (rest of Table \ref{rays}) to get a nonsingular projective equivariant compactification $X$, given by the cones of Table \ref{cones}, where just one cone is listed from each orbit under $G\times \{\pm 1\}$ (action of $G$ and multiplication by $\pm 1$ on $N_{\bR}$).
The other non-cyclic abelian subgroup $\langle x^2,xy\rangle$ fixes points in the chart with coordinates $\alpha$, $\beta$, $\gamma^{-1}$, $\gamma^{-1}\delta$, $\gamma\varepsilon$,
corresponding to the fourth cone listed in Table \ref{cones},
so the action satisfies Condition \textbf{(A)}.
We determine $\rH^1(G,\Pic(X))\cong \bZ/2\bZ$.

The computation of $\partial(1_{\Pic(X)})$, via Proposition \ref{prop.delidentity}, is done using isomorphisms for finitely generated torsion-free $\Lambda$ (cf.\ \cite[\S 2.1]{KT-dp})
\[ \rH^j(G,\Lambda\otimes k^\times)\cong \rH^j(G,\Lambda\otimes \bQ/\bZ)\eqto \rH^{j+1}(G,\Lambda)\qquad (j\ge 1), \]
with up-to-sign compatibility of connecting homomorphisms by Lemma \ref{lem.ABC2}:
\eqref{eqn.explicitrho} yields
$
((0,0,1,2,1),(0,0,1,0,1),(0,0,0,0,1))\in
(M^\vee)^3$,
which by the connecting homomorphism goes
to the class
of the element of $(\Pic(X)^\vee)^4$ that maps to
\[ (2e_7^\vee-2e_9^\vee+e_{19}^\vee-e_{20}^\vee,2e_7^\vee+2e_{15}^\vee-e_{19}^\vee-e_{20}^\vee-2e_{21}^\vee,0,0)\in (P^\vee)^4, \]
with respect to the dual basis to the basis $e_j$, $j=1$, $\dots$, $42$, of $P$, corresponding to the numbering of the rays (Table \ref{rays}).
The nontriviality of the class in $\rH^3(G,\Pic(X)^\vee)$ is confirmed by direct computation; the class is $2$-torsion:
\begin{align}
\begin{split}
\label{eqn.twotorsion}
\begin{pmatrix}1-x\\1+y\end{pmatrix}
(3&e_7^\vee-e_9^\vee+e_{15}^\vee+e_{17}^\vee-2e_{20}^\vee-2e_{21}^\vee)=\\
&\qquad\qquad\begin{pmatrix}4e_7^\vee-4e_9^\vee+2e_{19}^\vee-2e_{20}^\vee\\4e_7^\vee+4e_{15}^\vee-2e_{19}^\vee-2e_{20}^\vee-4e_{21}^\vee\end{pmatrix}.
\end{split}
\end{align}
Theorem \ref{thm:main} excludes $G$-unirationality.
\end{exam}

\begin{table}[H]
\small
\[ \setlength{\arraycolsep}{2pt}
\begin{array}{llll}
&\,\,\,1(1\,0\,0\,0\,0)&\,\,\,2({-1}\,0\,0\,0\,0)\\
\,\,\,3(0\,1\,0\,0\,0)&\,\,\,4(0\,{-1}\,0\,0\,0)&\,\,\,5(1\,{-1}\,0\,0\,0)&\,\,\,6({-1}\,1\,0\,0\,0)\\
\,\,\,7(0\,0\,0\,1\,0)&\,\,\,8(0\,0\,0\,{-1}\,0)&\,\,\,9(0\,0\,1\,1\,0)&10(0\,0\,{-1}\,{-1}\,0)\\
11(0\,1\,1\,0\,0)&12(0\,{-1}\,{-1}\,0\,0)&13(1\,{-1}\,{-1}\,0\,0)&14({-1}\,1\,1\,0\,0)\\
15(0\,0\,0\,{-1}\,1)&16(0\,0\,0\,1\,{-1})&17(0\,0\,{-1}\,{-1}\,1)&18(0\,0\,1\,1\,{-1})\\
19(0\,0\,1\,0\,0)&20(0\,0\,{-1}\,0\,0)&21(0\,0\,0\,0\,1)&22(0\,0\,0\,0\,{-1})\\
23(0\,0\,1\,2\,{-1})&24(0\,0\,{-1}\,{-2}\,1)&25({-1}\,2\,1\,0\,0)&26(1\,{-2}\,{-1}\,0\,0)\\
27(0\,1\,0\,{-1}\,0)&28(0\,{-1}\,0\,1\,0)&29(0\,1\,1\,1\,{-1})&30(0\,{-1}\,{-1}\,{-1}\,1)\\
31({-1}\,1\,1\,1\,0)&32(1\,{-1}\,{-1}\,{-1}\,0)&33(1\,{-1}\,0\,1\,{-1})&34({-1}\,1\,0\,{-1}\,1)\\
35(0\,1\,1\,1\,0)&36(0\,{-1}\,{-1}\,{-1}\,0)&37(1\,{-1}\,0\,1\,0)&38({-1}\,1\,0\,{-1}\,0)\\
39(0\,1\,0\,{-1}\,1)&40(0\,{-1}\,0\,1\,{-1})&41({-1}\,1\,1\,1\,{-1})&42(1\,{-1}\,{-1}\,{-1}\,1)
\end{array}
\]
\caption{Ray generators for $5$-dimensional fan}
\label{rays}
\end{table}

\begin{table}[H]
\tiny
\begin{gather*}
[1\,3\,7\,13\,16]\,\,[1\,3\,7\,13\,17]\,\,[1\,3\,7\,16\,23]\,\,[1\,3\,7\,17\,21]\,\,[1\,3\,7\,21\,35]\,\,[1\,3\,7\,23\,35]\,\,[1\,3\,10\,13\,16]\,\,[1\,3\,10\,16\,22]\\
[1\,3\,10\,17\,24]\,\,[1\,3\,10\,22\,27]\,\,[1\,3\,10\,24\,27]\,\,[1\,3\,11\,21\,35]\,\,[1\,3\,11\,22\,27]\,\,[1\,3\,11\,23\,29]\,\,[1\,3\,11\,23\,35]\\
[1\,3\,16\,22\,29]\,\,[1\,3\,16\,23\,29]\,\,[1\,3\,17\,21\,39]\,\,[1\,3\,17\,24\,39]\,\,[1\,7\,9\,21\,35]\,\,[1\,7\,9\,23\,35]\,\,[1\,8\,10\,22\,27]\\
[1\,8\,10\,24\,27]\,\,[3\,6\,7\,16\,20]\,\,[3\,6\,7\,16\,25]\,\,[3\,6\,7\,17\,20]\,\,[3\,6\,7\,17\,25]\,\,[3\,7\,13\,16\,20]\,\,[3\,7\,13\,17\,20]\\
[3\,7\,16\,23\,25]\,\,[3\,7\,17\,21\,25]\,\,[3\,7\,21\,25\,35]\,\,[3\,7\,23\,25\,35]\,\,[3\,10\,13\,16\,20]\,\,[3\,10\,16\,22\,25]\,\,[3\,10\,17\,24\,25]\\
[3\,10\,22\,25\,27]\,\,[3\,10\,24\,25\,27]\,\,[3\,11\,21\,25\,35]\,\,[3\,11\,22\,25\,27]\,\,[3\,11\,23\,25\,29]\,\,[3\,11\,23\,25\,35]\,\,[3\,16\,22\,25\,29]\\
[3\,16\,23\,25\,29]\,\,[3\,17\,21\,25\,39]\,\,[3\,17\,24\,25\,39]\,\,[7\,9\,21\,25\,31]\,\,[7\,9\,21\,25\,35]\,\,[7\,9\,23\,25\,31]\,\,[7\,9\,23\,25\,35]
\end{gather*}
\caption{Fan with $384$ cones of top dimension obeys symmetry by group action and $-1$; list of orbit representatives}
\label{cones}
\end{table}

\begin{rema}
\label{rem.D8}
In Example \ref{exa.noAm3}, if we let
$\widetilde{G}:=\mathfrak{D}_8$, with generators $\tilde x$ of order $8$ and $\tilde y$ of order $2$, act on $X$ via the
homomorphism $\mathfrak{D}_8\to \mathfrak{D}_4$,
$\tilde x\mapsto x$, $\tilde y\mapsto y$,
then $X$ is $\widetilde{G}$-unirational.
By a completely analogous computation, or by functoriality of the Leray spectral sequence, we find the class in $\rH^3(\widetilde{G},\Pic(X)^\vee)$ given by twice the element of $(\Pic(X)^\vee)^4$ obtained above (with an analogous complex for group cohomology of $\widetilde{G}$); this is trivial by \eqref{eqn.twotorsion}.
\end{rema}

\section{Projective unirationality}
\label{scn.projectiveunirationality}
In \cite{tz}, cohomological obstructions are explored not only to equivariant unirationality, but also to a new, related condition on $G$-actions, defined as follows.
A $G$-action is
\begin{itemize}
\item[({\bf PU})] {\em projectively unirational} if there exists a central cyclic extension $\widetilde{G}$ of $G$, such that $X$ is $\widetilde{G}$-unirational. 
\end{itemize}

We have the following cohomological obstruction:
\begin{itemize}
\item If the $G$-action is projectively unirational, then $\mathrm{Am}^2(X,H)$ is cyclic and $\mathrm{Am}^3(X,H)=0$, for all $H\subseteq G$.
\end{itemize}

\begin{exam}
\label{exam:P1continued}
In Example \ref{exam:p1}, the action of the Klein four-group $C_2^2$ on $\bP^1$ 
is projectively unirational, as is that of any $G$ via $\lambda\colon G\to \PGL_2$ with image the Klein four-group.
However, the product action of $C_2^2\times C_2^2$ on $\bP^1\times \bP^1$ has noncyclic $\mathrm{Am}^2$, hence is not projectively unirational.
\end{exam}

We wish to reformulate the condition \textbf{(PU)} entirely in terms of $G$-actions.
For instance, a projective representation $G\to \PGL_n$
determines a $G$-action on $\bP^{n-1}$.
The condition, to be equivariantly rationally dominated by this $\bP^{n-1}$, is one that is phrased entirely in terms of $G$-equivariant geometry.
We show that with finitely many choices of projective representation, each allowed in the form of a sum of arbitrarily many copies, we obtain an equivalent condition to \textbf{(PU)}.

We start with an easy observation.
Given a central cyclic extension
\begin{equation}
\label{eqn.cce}
1\to Z\to \widehat{G}\to G\to 1
\end{equation}
and identification of $Z$ with $\langle \zeta\rangle$, where $\zeta\in k^\times$ is a primitive $\ell$th root of unity, $\ell=|Z|$,
there exists a representation $W$ of $\widehat{G}$ of positive dimension, such that the subgroup of $\widehat{G}$ acting by scalars on $W$ is precisely $Z$, with action given by the identification with $\langle\zeta\rangle$.
For instance, we may decompose the regular representation $\widehat{V}$ of $\widehat{G}$ as
\[ \widehat{V}=\bigoplus_{j\in \bZ/\ell\bZ} \widehat{V}_j, \]
according to type upon restriction to $Z$,
and take $W=\widehat{V}_1$.

A central cyclic extension \eqref{eqn.cce}, with identification $Z\cong \langle \zeta\rangle$, determines a class in
$\rH^2(G,k^\times)$, the obstruction to existence of a splitting when we allow to replace $\ell$ by a suitable multiple, cf.\ \cite[\S 2]{KT-Brauer}.

\begin{prop}
\label{prop.PU}
Let $X$ be a smooth projective rational variety over $k$ with regular action of a finite group $G$.
Let us fix $\gamma\in \rH^2(G,k^\times)$, a central cyclic extension \eqref{eqn.cce} of class $\gamma$, and a representation $W$ of positive dimension, where the subgroup acting by scalars is $Z$, acting by $Z\cong \langle\zeta\rangle$, $\zeta\in k^\times$.
The following are equivalent.
\begin{itemize}
\item[(i)] There exists a central cyclic extension $\widetilde{G}$ of $G$ such that $X$ is $\widetilde{G}$-unirational, where $\widetilde{G}$ has class $\gamma$ for a suitable identification of $\ker(\widetilde{G}\to G)$ with roots of unity in $k$.
\item[(ii)] For some positive integer $m$ there exists a dominant $G$-equivariant rational map $\bP((W^{\oplus m})^\vee)\dashrightarrow X$.
\end{itemize}
\end{prop}

\begin{proof}
We suppose $G$ nontrivial.
Clearly, (ii) $\Rightarrow$ (i).
For (i) $\Rightarrow$ (ii), we use the
theory of projective representations \cite[Chap.\ 11]{isaacs} to
reduce to showing that (i), with given extension $\widetilde{G}$, implies the case $\widehat{G}=\widetilde{G}$ of (ii).

To establish (ii) it suffices to establish the
existence of a dominant $\widehat{G}$-equivariant rational map $(W^{\oplus m})^\vee\dashrightarrow X$.
Indeed, by the
equivariant birational identification with
$\mathcal{O}_{\bP((W^{\oplus m})^\vee)}(-1)$ we obtain
from $Z$-invariance an induced
$G$-equivariant rational map
from $\mathcal{O}_{\bP((W^{\oplus m})^\vee)}(-\ell)$.
Then the No-Name Lemma (cf.\ \cite[\S 2]{KT-Brauer}) yields
$\bP((W^{\oplus m})^\vee)\times \bP^1\dashrightarrow X$,
where $G$ acts trivially on the factor $\bP^1$.
With any nonconstant invariant rational function on $\bP(W^\vee)$ we get a $G$-equivariant dominant rational map to $X$ from
$\bP((W^{\oplus m})^\vee)\times \bP(W^\vee)$, hence also from $\bP((W^{\oplus (m+1)})^\vee)$.

It remains to show that $\widehat{G}$-unirationality implies the existence of a dominant $\widehat{G}$-equivariant rational map $(W^{\oplus m})^\vee\dashrightarrow X$.
This follows from the classical result of Burnside, that the tensor powers of a faithful representation contain all irreducible representations
(see \cite{steinberg}), by covering the projectivization of a representation by the representation, then by a sum of tensor powers of $W^\vee$, and finally by a sum of copies of $W^\vee$.
\end{proof}

For each $\gamma\in \rH^2(G,k^\times)$ we fix $\widehat{G}_\gamma$ and $W_\gamma$ as in Proposition \ref{prop.PU}.

\begin{coro}
\label{cor.PU}
For a smooth projective rational variety $X$ over $k$ with regular action of a finite group $G$, the following are equivalent.
\begin{itemize}
\item[(i)] The $G$-action on $X$ is projectively unirational.
\item[(ii)] The $G$-action on $X$ is $\widehat{G}_\gamma$-unirational for some $\gamma\in \rH^2(G,k^\times)$.
\item[(iii)] For some $\gamma\in \rH^2(G,k^\times)$ and positive integer $m$, there exists a dominant $G$-equivariant rational map $\bP((W_\gamma^m)^\vee)\dashrightarrow X$.
\end{itemize}
\end{coro}

\begin{rema}
\label{rem.PUtoric}
When $X$ is a toric variety, with $G$-action preserving the torus, condition (ii) of Corollary \ref{cor.PU} can be tested effectively by applying Theorem \ref{thm:main} to each of the groups $\widehat{G}_\gamma$.
\end{rema}

\begin{rema}
\label{rem.PUsubtlety}
Consider a $G$-action on $X$ which is projectively unirational but not unirational.
For $\gamma\in \rH^2(G,k^\times)$
in Corollary \ref{cor.PU} (ii)--(iii), we must have
$\mathrm{Am}^2(X,G)\subseteq \langle\gamma\rangle$, with $\gamma\ne 0$.
One might ask, could Corollary \ref{cor.PU} be strengthened by imposing the additional condition
$\mathrm{Am}^2(X,G)=\langle\gamma\rangle$
in (ii)--(iii)?
The answer is no.
In Example \ref{exa.noAm3} we have
such a $G$-action (Remark \ref{rem.D8}), with $\mathrm{Am}^2(X,G)=0$.
\end{rema}

\appendix
\section{Remarks on homological algebra}
\label{app.homologicalalgebra}
Let $\mathcal{A}$ be an abelian category with enough injectives.
In this appendix we record some compatibilities among
$\Ext^i$ and spectral sequences.

\subsection{Extension class}
\label{ss.extensionclass}
Let
\begin{equation}
\label{eqn.ABC}
0\to A\stackrel{u}\to B\stackrel{v}\to C\to 0
\end{equation}
be a short exact sequence in $\mathcal{A}$.
We start by recalling the extension class in a classical setting, where $\mathcal{A}$ also has enough projectives.
For an injective resolution $(I^\bullet)$ of $A$ with $\iota_A\colon A\to I^0$ and projective resolution $(P_\bullet)$ of $C$ with $\varepsilon_C\colon P_0\to C$,
we extend $\iota_A$ to $\beta\colon B\to I^0$ and lift $\varepsilon_C$ to $\varphi\colon P_0\to B$.
The extension class of \eqref{eqn.ABC}
is the class of the map $\ker(\varepsilon_C)\to A$ induced by $\varphi$, in $\Hom(\ker(\varepsilon_C),A)/\Hom(P_0,A)$, which is also the class of the map $C\to \coker(\iota_A)$ induced by $\beta$, in an analogous cokernel of $\Hom$'s.
The respective cokernels are identified in the standard way (cf.\ \S\ref{ss.groupcohomology}) and also directly (and compatibly) via the
expression
\[
\Ext^1(C,A)\cong \frac{\ker\bigl(\Hom(P_0,I^0)\stackrel{d\circ(\,)\circ \partial}\longrightarrow \Hom(P_1,I^1)\bigr)}{\ker(d\circ(\,))+\ker((\,)\circ \partial)},
\]
given in \cite[\S III.3]{HS}.

Returning to the setting of abelian category with enough injectives, we have the extension class
\begin{equation}
\label{eqn.classABC}
\alpha\in \Ext^1(C,A),
\end{equation}
given by $C\to \coker(\iota_A)$ as above.

\begin{lemm}
\label{lem.ABCidentity}
We have connecting homomorphisms
\begin{align*}
\End(A)&\to \Ext^1(C,A),&1_A&\mapsto \alpha, \\
\End(C)&\to \Ext^1(C,A),&1_C&\mapsto -\alpha.
\end{align*}
\end{lemm}

\begin{proof}
The first one follows directly from application of $\Hom({-},I^\bullet)$.
We get the second from an injective resolution $(K^\bullet)$ of $C$ and
compatible injective resolution $(J^\bullet)$ of $B$, with $J^n=I^n\oplus K^n$, by $\Hom(C,{-})$.
\end{proof}

\subsection{Derived category}
\label{ss.derivedcategory}
The derived category supplies a further perspective.
We consider a short exact sequence
of complexes
\[
0\to A^\bullet\stackrel{u}\to B^\bullet\stackrel{v}\to C^\bullet\to 0
\]
in $\mathcal{A}$, where \eqref{eqn.ABC} may be recovered by taking the complexes to consist of a single object in degree $0$.
There is an associated distinguished triangle
\[ A^\bullet\stackrel{u}\to B^\bullet\stackrel{v}\to C^\bullet\to A^\bullet[1] \]
in the derived category $\mathbf{D}(\mathcal{A})$.
Some conventional choices are required; we follow \cite{conrad}, which gives
$C\to A^\bullet[1]$ as $-p$ composed with the formal inverse to the quasi-isomorphism $q$:
\[
\xymatrix{
\mathrm{cone}^\bullet(u) \ar[d]_{q} \ar[r]^{-p} & A^\bullet[1] \\
C^\bullet
}
\]
with
\begin{align*}
\mathrm{cone}^\bullet(u)&=\biggl(\dots\to A^{n+1}\oplus B^n
\stackrel{\setlength\arraycolsep{2pt}\begin{pmatrix}\scriptstyle-d_A&\scriptstyle0\\ \scriptstyle u&\scriptstyle d_B\end{pmatrix}}{\relbar\joinrel\longrightarrow}
A^{n+2}\oplus B^{n+1}\to \dots\biggr), \\
p(x,y)&=x, \\
q(x,y)&=v(y).
\end{align*}

When we specialize to \eqref{eqn.ABC},
with previous notation, we let $\gamma\colon C\to I^1$ denote the unique map with $\gamma\circ v=-d\circ \beta$ and get a commutative diagram
\[
\xymatrix{
(A\stackrel{u}\to B) \ar[r]^{-p} \ar[d]_q & (A\to 0)\ar[d]^{\iota_A} \\
(0\to C)\ar[r]^(0.4){\gamma} & (I^0\stackrel{-d}\to I^1\stackrel{-d}\to \dots)
}
\]
in $\mathbf{D}(\mathcal{A})$.
This gives the alternate description of $C\to A[1]$ as $(\iota_A)^{-1}\circ \gamma$, which is
\[ -\alpha\in \Ext^1(C,A)=\Hom_{\mathbf{D}(\mathcal{A})}(C,A[1]), \]
due to the minus sign in the definition of $\gamma$.
So the derived category morphism $C\to A[1]$, associated with the short exact sequence \eqref{eqn.ABC}, is \emph{inverse} to the class \eqref{eqn.classABC}.

The choice of definition of $\gamma$ is a convenient one, since we extend $\gamma$ to $\eta\colon K^0\to I^1$, where $(K^\bullet)$ is an injective resolution of $C$ with $\iota_C\colon C\to K^0$, and proceed iteratively by observing that $-d\circ \eta$ induces a map on the cokernel of $\iota_C$, which extends to $\eta\colon K^1\to I^2$, and so on, to obtain
\[ \eta\colon K^\bullet\to I^\bullet[1]. \]
The computation of the second assertion of Lemma \ref{lem.ABCidentity} may be neatly organized, using
\[
J^n=I^n\oplus K^n,\qquad J^n
\stackrel{\setlength\arraycolsep{2pt}\begin{pmatrix}\scriptstyle d&\scriptstyle \eta\\ \scriptstyle0&\scriptstyle d\end{pmatrix}}{\relbar\joinrel\longrightarrow}
J^{n+1},
\]
with augmentation map is $(\beta,\iota_C\circ v)\colon B\to J^0$.

We record a standard compatibility of connecting homomorphisms (cf.\ \cite[\S VI.1]{cartaneilenberg}).

\begin{lemm}
\label{lem.ABC2}
Let, additionally, a short exact sequence
\[ 0\to A'\to B'\to C'\to 0 \]
be given.
With connecting homomorphisms
\[
\xymatrix{
\Ext^i(A',C) \ar[r]^{\delta} \ar[d]_{\delta'} &
\Ext^{i+1}(A',A) \ar[d]^{\delta'} \\
\Ext^{i+1}(C',C) \ar[r]^{\delta} &
\Ext^{i+2}(C',A)
}
\]
we have $\delta\circ\delta'+\delta'\circ\delta=0$.
\end{lemm}

\begin{proof}
We consider an element of
$\Ext^i(A',C)$, represented by
$\rho\colon A'\to K^i$.
To determine $\delta'([\rho])$ we
extend $\rho$ to $\sigma\colon B'\to K^i$, then $d\circ \sigma$ induces
$\tau\colon C'\to K^{i+1}$, and $\delta'([\rho])=[\tau]$.
So $\delta(\delta'([\rho]))=[\eta\circ\tau]$.
We have the extension $\eta\circ\sigma$ of $\eta\circ\rho$, and we obtain $$
\delta'(\delta([\rho]))=\delta'([\eta\circ\rho])=-[\eta\circ\tau]
$$
from $-d\circ \eta=\eta\circ d$.
\end{proof}

\subsection{Group cohomology}
\label{ss.groupcohomology}
Let $G$ be a finite group.
For the abelian category of $G$-modules we make frequent use of the standard isomorphism
\[
\Ext^n(M,N)\cong\rH^n(G,M^\vee\otimes N),
\]
when $M$ is a finitely generated torsion-free $G$-module.

Generally, for an injective resolution $(I^\bullet)$ of $A$ and projective resolution $(P_\bullet)$ of $C$, we may obtain
\[ \Ext^n(C,A)=\rH^n(\mathbf{R}\Hom^\bullet(C,A)) \]
as cohomology in degree $n$ of the complex
\[
\xymatrix@R=10pt{
\Hom(P_0,I^0) \ar[r]\ar@{}@<-2pt>[r]^{-} \ar[dr] &
\Hom(P_1,I^0) \ar@{}[d]|{\displaystyle\oplus} \ar[r] \ar[dr] &
\Hom(P_2,I^0) \ar@{}[d]|{\displaystyle\oplus} \ar[r]\ar@{}@<-2pt>[r]^(0.6){-}\ar[dr] & {\phantom{\cdots}} \\
& \Hom(P_0,I^1) \ar[r] \ar[dr] & \Hom(P_1,I^1) \ar@{}[d]|{\displaystyle\oplus} \ar[r]\ar@{}@<-2pt>[r]^(0.6){-} \ar[dr] & {\phantom{\cdots}} \\
& & \Hom(P_0,I^2) \ar[r]\ar@{}@<-2pt>[r]^(0.6){-} \ar[dr] & {\cdots} \\
& & & {\phantom{\cdots}}
}
\]
The slanted arrows are composition with $d$, and the horizontal ones, composition with $\pm \partial$, where the sign is $-$ when indicated.
This is identified with $\rH^n(\Hom(C,I^\bullet))$ and also with $\rH^n$ of
\[
\Hom(P_0,A)\stackrel{-}\longrightarrow \Hom(P_1,A)\longrightarrow \Hom(P_2,A)\stackrel{-}\longrightarrow \dots.
\]
The identification with the cohomology of $\Hom(P_\bullet,A)$, where the differential is composition with $\partial$ \emph{without} signs, is by multiplication by $(-1)^{n(n+1)/2}$ in degree $n$; cf.\ \cite[p.~13]{conrad}.

Although group cohomology $\rH^n(G,{-})$ is defined as $\mathrm{R}^n({-})^G$, it is sometimes described as $\rH^n$ of the complex arising by $\Hom_G$ from a free $\bZ[G]$-module resolution $(P_\bullet)$ of $\bZ$; see \cite[Rmk., \S III.1]{brown}.
The identification with $\Ext$ and comparison of connecting homomorphisms are as follows:
\begin{align*}
\Ext^n(M,N)&\cong \rH^n(\Hom_G(P_\bullet,M^\vee\otimes N))&\text{with sign}& \quad (-1)^{\frac{n(n+1)}{2}}, \\
\rH^n(\Hom_G(P_\bullet,C))&\to \rH^{n+1}(\Hom_G(P_\bullet,A))&\text{with sign}&\quad (-1)^{n+1},
\end{align*}
for $G$-modules $M$ and $N$, with $M$ finitely generated and torsion-free, and short exact sequence \eqref{eqn.ABC} of $G$-modules.

\subsection{Spectral sequence}
\label{ss.spectralsequence}
In applications of the
Leray spectral sequence we may need a $d_2$ map in explicit form.
More generally, we are interested in the hypercohomology spectral sequence, for a left exact functor from one abelian category to another.
For instance, \cite[Prop.\ 6.1]{KT-effect} relates a $d_2$ map to connecting homomorphisms in group cohomology.
The proof is by application of the truncation functor $\tau_{\le 1}$ and standard compatibility.
With the convention, following Grothendieck \cite{grothtohoku}, of double complexes with commuting squares, the
compatibility takes the following form.

\begin{lemm}
\label{lem.twoterm}
Let $(A^0\stackrel{g}\to A^1)$ be a $2$-term complex in $\mathcal{A}$, and let $F$ be a left exact functor to another abelian category.
In the spectral sequence $$
\mathrm E_2^{pq}=\mathrm{R}^p_{\phantom0}\!F(\rH^q_{\phantom0}(A^\bullet))\Rightarrow \mathbb{R}^{p+q}_{\phantom0}\!F(A^\bullet)
$$
the map
$d_2^{i1}\colon \mathrm{R}^i\!F(\coker(g))\to \mathrm{R}^{i+2}\!F(\ker(g))$ is equal to the composite
\[ \mathrm{R}^i\!F(\coker(g))\to \mathrm{R}^{i+1}\!F(\image(g))\to \mathrm{R}^{i+2}\!F(\ker(g)). \]
\end{lemm}

\begin{proof}
We have short exact sequences
\[
0\to D\stackrel{u}\to A^0\stackrel{v}\to B\to 0
\qquad\text{and}\qquad
0\to B\stackrel{\tilde u}\to A^1\stackrel{\tilde v}\to C\to 0,
\]
where $D$, $C$, $B$ denote the kernel, cokernel, and image of $g$.
We take $(I^\bullet)$, $(\widetilde{I}^\bullet)$, $(K^\bullet)$ to be injective resolutions of $D$, $C$, $B$, respectively.
To the first short exact sequence we associate a lift $\beta\colon A^0\to I^0$ of the augmentation map $\iota_D$ and $\eta\colon K^\bullet\to I^\bullet[1]$, as above.
Analogously, to the second we associate $\tilde\beta\colon A^1\to K^0$ and $\tilde\eta\colon \widetilde{I}^\bullet\to K^\bullet[1]$.
The complex
\begin{equation}
\label{eqn.IItilde}
I^0
\stackrel{\begin{pmatrix}\scriptstyle d\\ \scriptstyle 0\end{pmatrix}}\longrightarrow
I^1\oplus \widetilde{I}^0
\stackrel{\setlength\arraycolsep{2pt}\begin{pmatrix}\scriptstyle d&\scriptstyle \eta\circ\tilde\eta\\ \scriptstyle 0&\scriptstyle -d\end{pmatrix}}{\relbar\joinrel\relbar\joinrel\longrightarrow}
I^2\oplus \widetilde{I}^1
\stackrel{\setlength\arraycolsep{2pt}\begin{pmatrix}\scriptstyle d&\scriptstyle \eta\circ\tilde\eta\\ \scriptstyle 0&\scriptstyle -d\end{pmatrix}}{\relbar\joinrel\relbar\joinrel\longrightarrow}
\dots,
\end{equation}
is quasi-isomorphic to $(A^0\stackrel{g}\to A^1)$ by $\beta$ and $(-\eta\circ\tilde\beta,\iota_C\circ \tilde v)$.
For \eqref{eqn.IItilde} we have a Cartan-Eilenberg resolution $(L^{{\bullet}{\bullet}})$ with $L^{0q}=I^q\oplus I^{q+1}$ and
$L^{pq}=I^{p+q}\oplus \widetilde{I}^{p+q-1}\oplus I^{p+q+1}\oplus \widetilde{I}^{p+q}$ for $p>0$.
By direct computation, $d_2^{i1}$ is found to be $F(\eta\circ\tilde\eta)=F(\eta)\circ F(\tilde\eta)$, and this is the composite of the connecting homomorphisms from the statement of the lemma.
\end{proof}

\bibliographystyle{plain}
\bibliography{coho-unirat}

\end{document}